\documentclass[runningheads]{llncs}

\usepackage[T1]{fontenc}
\usepackage{graphicx}
\usepackage{algorithm,algorithmic}
\usepackage{amsmath,amssymb,amsfonts}
\usepackage{hyperref}
\usepackage{listings}
\usepackage{xifthen}

\newcommand*{\mathlibdef}[2][]{\href{https://leanprover-community.github.io/mathlib4_docs/find/?pattern=#2\#src}{\ifthenelse{\isempty{#1}}{\lean{#2}}{#1}}}
\newcommand*{\ourdef}[2][]{\href{https://wuprover.github.io/lean_characteristic_set/docs/find/?pattern=#2\#src}{\ifthenelse{\isempty{#1}}{\lean{#2}}{#1}}}

\lstnewenvironment{leancode}{\lstset{frame=single,language=lean,abovecaptionskip=-\medskipamount,basicstyle={\ttfamily\scriptsize}}}{}
\newcommand{\lean}[1]{\texttt{#1}}

\usepackage{todonotes}
\usepackage{color}
\definecolor{keywordcolor}{rgb}{0.7, 0.1, 0.1}   % red
\definecolor{tacticcolor}{rgb}{0.0, 0.1, 0.6}    % blue
\definecolor{commentcolor}{rgb}{0.4, 0.4, 0.4}   % grey
\definecolor{symbolcolor}{rgb}{0.0, 0.1, 0.6}    % blue
\definecolor{sortcolor}{rgb}{0.1, 0.5, 0.1}      % green
\definecolor{attributecolor}{rgb}{0.7, 0.1, 0.1} % red

\newcommand{\init}{\operatorname{init}}
\newcommand{\remdr}{\operatorname{Remdr}}
\newcommand{\zero}{\operatorname{Zero}}

\newtheorem{themcounter}{Theorem}
\newtheorem{theorem}[themcounter]{Theorem}
\newtheorem{proposition}[themcounter]{Proposition}
\newtheorem{corollary}[themcounter]{Corollary}

\begin{document}

\title{Formalizing Wu-Ritt Method in Lean~4}


\author{Yuxuan Xiao \inst{1,2} \and
Hao Shen\inst{2} \and
Junyu Guo \inst {3} \and 
Dingkang Wang\inst{2} \and
Lihong Zhi \inst{2}
\authorrunning{F. Author et al.}
\institute{School of Mathematics, Shandong University, Shandong, 250100, China
\email{yuxuanxiao@mail.sdu.edu.cn} % xiao-yuxuan@qq.com
\and
State Key Laboratory of Mathematical Sciences, Academy of Mathematics and Systems Science, University of Chinese Academy of Science, Beijing, 100190, China 
\email{\{shenhao24,dwang,lzhi\}@amss.ac.cn}
\and
Institute of Logic and Cognition, Department of Philosophy, Sun Yat-sen University, 510275, Guangdong, China \\
\email{guojy228@mail2.sysu.edu.cn}}}

%Springer Heidelberg, Tiergartenstr. 17, 69121 Heidelberg, Germany

%
\maketitle              % typeset the 




%\maketitle              % typeset the header of the contribution

\begin{abstract}

We formalize the Wu–Ritt characteristic set method for the triangular decomposition of polynomial systems in the Lean~4 theorem prover. Our development includes the core algebraic notions of the method, such as polynomial initials,  orders, pseudo-division, pseudo-remainders with respect to a polynomial or a triangular set, and standard and weak ascending sets. On this basis, we formalize algorithms for computing basic sets, characteristic sets, and zero decompositions, and prove their termination and correctness. In particular, we formalize the well-ordering principle relating a polynomial system to its characteristic set and verify that zero decomposition expresses the zero set of the original system as a union of zero sets of triangular sets away from the zeros of the corresponding initials. This work provides a machine-checked verification of Wu-Ritt’s method in Lean~4 and establishes a foundation for certified polynomial system solving and geometric theorem proving.


\end{abstract}

\section{Introduction}\label{sec1}

The Wu-Ritt characteristic set method, introduced by Wu Wen-Tsun~\cite{wen1986basic}, is a fundamental tool for solving systems of multivariate polynomial equations and for automated theorem proving in elementary geometry. 

In this paper, we present a Lean~4 formalization of the Wu-Ritt characteristic set method~\cite{lean4prover}, following Wu’s mathematical and algorithmic treatment~\cite{wen2001mathematics}. Our development formalizes the basic concepts of the theory, including pseudo-division, pseudo-remainder with respect to a polynomial or a triangular set, ascending sets, basic sets, and characteristic sets. We verify not only the main algorithms of the method but also the central statements of the Wu--Ritt theory. In particular, we formalize the well-ordering principle that connects a polynomial system with its characteristic set, and we prove the correctness of the zero decomposition algorithm, showing that it expresses the zero set of the original system as a finite union of zero sets of triangular sets, away from the zeros of the corresponding initials.

A formalization of Wu's simple method was previously carried out in Coq~\cite{coqmanual83} for proving geometric theorems \cite{formalwucoq}. Our work provides a comprehensive formalization of the characteristic set method in Lean~4. It builds on Mathlib’s algebraic infrastructure and provides a basis for certified symbolic computation and automated geometry theorem proving in the Lean ecosystem~\cite{mathlib}.

\section{Preliminaries}\label{sec2}

% \subsection{Polynomial Order}

Let \(R\) be a commutative ring with identity, and let \(\sigma\) be a finite set of indices equipped with a linear order \(<\). We write
\(
R[X_i]_{i\in\sigma}
\)
for the polynomial ring under consideration, ordering the variables by declaring
\[
X_i < X_j \quad \text{whenever } i < j.
\]

Let \(p \in R[X_i]_{i\in \sigma}\) be a polynomial.
The greatest variable \(X_m\) appearing in \(p\) is called the \textit{main variable} of \(p\). We define the \textit{main degree} of \(p\), denoted by \(\deg(p)\), to be the degree of \(p\) with respect to its main variable, namely \(\deg_m(p)\). If \(p\) is a constant polynomial, then its main variable is undefined, and we define \(\deg(p)=0\).

We first introduce a collection of basic definitions.

\begin{definition}[{\ourdef[Polynomial Order]{MvPolynomial.order}}]
\label{polyOrder}
Let $p$ and $q$ be two non-zero multivariate polynomials with main variables $X_{m_p}, X_{m_q}$, respectively.
We say $p$ is less than $q$, denoted by $p \prec q$, if
either $m_p < m_q$ or $m_p = m_q$ and $\deg(p) < \deg(q)$.
\end{definition}

In Lean~4, this can be defined as a lexicographical order. The order here is based on the order of the indices $\sigma$, rather than the monomial order used in the formalization of Gröbner bases \cite{guo2026formalizinggrobnerbasistheory}.

\begin{leancode}
def order (p : MvPolynomial σ R) : WithBot σ ×ₗ ℕ := (p.vars.max, p.mainDegree)
\end{leancode}

We say that \(p\) and \(q\) have the same ordering, denoted by \(p \sim q\), if neither \(p \prec q\) nor \(q \prec p\) holds. Equivalently,
\[
p \sim q \quad \Longleftrightarrow \quad m_p = m_q \ \land\ \deg(p)=\deg(q).
\]
We write \(p \preceq q\) if \(p \prec q\) or \(p \sim q\). The relation \(\preceq\) defines a preorder on \(R[X_i]_{i\in \sigma}\), and \(\sim\) is the corresponding equivalence relation. Hence \(\preceq\) induces a partial order on the quotient \(R[X_i]_{i\in \sigma} / {\sim}\).

\begin{leancode}
instance instPreorder : Preorder (MvPolynomial σ R) where
  le := InvImage (· ≤ ·) order
  le_refl := fun _ ↦ by rw [InvImage]
  le_trans := fun _ _ _ ↦ le_trans

theorem lt_def : p < q ↔ p.vars.max < q.vars.max ∨
    p.vars.max = q.vars.max ∧ p.mainDegree < q.mainDegree

instance instSetoid : Setoid (MvPolynomial σ R) := AntisymmRel.setoid (MvPolynomial σ R) (· ≤ ·)
\end{leancode}

\begin{definition}[{\ourdef[Initial]{MvPolynomial.initial}}]
\label{initial}
% The initial of a polynomial $p$ with respect to $i \in \sigma$ is the leading coefficient of $p$ viewed as a univariate polynomial in
% $R'[X_i]$ where $R' = R[X_j]_{j \ne i}$, denoted by $\init_i(p)$.
Let \(p\) be a nonzero polynomial with main variable \(X_i\). Writing \(p\) as a polynomial in \(X_i\),
\[
p = c_d X_i^d + c_{d-1} X_i^{d-1} + \cdots + c_0,
\]
where each \(c_k \in R[X_j]_{j < i}\) and \(c_d \ne 0\), we call \(c_d\) the \textit{initial} of \(p\), and denote it by \(\init(p)\) or simply $I_p$. 
We define the initial of the zero polynomial to be \(0\). If \(p\) is a nonzero constant polynomial, then we set \(I_p = 1\).
\end{definition}

\begin{leancode}
def initialOf (p : MvPolynomial σ R) (i : σ) : MvPolynomial σ R :=
  ∑ s ∈ p.support with s i = p.degreeOf i, monomial (s.erase i) (p.coeff s)

def initial (p : MvPolynomial σ R) : MvPolynomial σ R :=
  if p = 0 then 0 else
    match p.vars.max with
    | ⊥ => 1
    | some m => p.initialOf m
\end{leancode}

\begin{leancode}
theorem initialOf_eq_leadingCoeff [DecidableEq σ] {p : MvPolynomial σ R} {i : σ} :
    p.initialOf i = rename Subtype.val
      (optionEquivLeft R {b // b ≠ i} (rename (Equiv.optionSubtypeNe i).symm p)).leadingCoeff
\end{leancode}

The \textit{initial product} of polynomials \(p_1, p_2, \ldots, p_r\) is defined to be the product of their initials $\prod_{i=1}^r I_{p_i}$. 

\begin{example}
Consider $p = 5 x y^2 + x^3 y, q = y^2 + x^2 y, r = x^5 \in \mathbb{Z}[x,y]$ with $x < y$.
Then $r \prec p$ and $ p \sim q$, $I_p = \init_y(p) = 5 x, I_q = 1$, and $I_r = 1$.
\end{example}

\begin{definition}[{\ourdef[Reduce]{MvPolynomial.reducedTo}}]
\label{reduce}
Let \(p\) be a nonzero polynomial with main variable \(X_m\). A polynomial \(q\) is said to be \textit{reduced with respect to} \(p\) if the degree of \(q\) in \(X_m\) is strictly smaller than the main degree of \(p\). %, namely,
%$\deg_{X_m}(q) < \deg(p).$
By convention, the zero polynomial is reduced with respect to every polynomial. If \(p\) is a constant polynomial, then no nonzero polynomial is reduced with respect to \(p\).
\end{definition}

\begin{leancode}
def reducedTo (q p : MvPolynomial σ R) : Prop :=
  if q = 0 then True
  else
    match p.vars.max with
    | ⊥ => False
    | some m => q.degreeOf m < p.degreeOf m
\end{leancode}

Let \(PS\) be a set of polynomials. We say that a polynomial \(q\) is \textit{reduced with respect to} \(PS\) if it is reduced with respect to every polynomial in \(PS\).

\section{Triangular Set}

\begin{definition}[{\ourdef[Triangular Set]{TriangularSet}}]
\label{triSet}
A set of polynomials \(S\) is called a \textit{triangular set} if all its elements are nonzero and can be arranged as a sequence
\[
s_1, s_2, \ldots, s_r \in R[X_i]_{i\in \sigma},
\]
% whose main variables satisfy
% \[
% X_{m_1} < X_{m_2} < \cdots < X_{m_r}.
% \]
whose main variables \(X_{m_1}, X_{m_2}, \ldots, X_{m_r}\) satisfy
\[
m_1 < m_2 < \cdots < m_r.
\]
\end{definition}

\begin{leancode}
structure TriangularSet (σ R : Type*) [CommSemiring R] [LinearOrder σ] where
  length' : ℕ
  /-- The sequence of polynomials, indexed by `ℕ`. (0-based) -/
  seq : ℕ → MvPolynomial σ R
  elements_ne_zero : ∀ n, n < length' ↔ seq n ≠ 0
  /-- The main variables of the polynomials are strictly increasing. -/
  ascending_max_vars : ∀ n < length' - 1, (seq n).vars.max < (seq (n + 1)).vars.max
\end{leancode}

Especially, our formalization allows $s_1$ to be a nonzero constant. In this case, $m_1$ is undefined, while $m_2 < \ldots < m_r$.

When \(S\) is a triangular set, we always write it as \(s_1,\dots,s_r\) so that the main variables of the \(s_i\) are strictly increasing with \(i\).

\begin{definition}[{\ourdef[Ordering of Triangular Sets]{TriangularSet.order}}]
\label{triSetOrder}
Let \(S=(s_1,\dots,s_r)\) and \(S'=(s'_1,\dots,s'_{r'})\) be two triangular sets. We say that \(S\) is \textit{less than} \(S'\), and write \(S \prec S'\), if one of the following conditions holds:
\begin{enumerate}
    \item there exists \(k \le \min(r,r')\) such that
    \[
    s_1 \sim s'_1,\; s_2 \sim s'_2,\; \ldots,\; s_{k-1} \sim s'_{k-1},
    \]
    while
    \[
    s_k \prec s'_k;
    \]
    \item \(r > r'\) and \(s_i \sim s'_i\) for all \(1 \le i \le r'\).
\end{enumerate}
\end{definition}

\begin{leancode}
-- The indices here are 0-based.
def order (S : TriangularSet σ R) : Lex (ℕ → WithTop (WithBot σ ×ₗ ℕ)) :=
  fun i ↦ if i < S.length then WithTop.some (S i).order else ⊤
\end{leancode}

\begin{leancode}
theorem lt_def : S < T ↔ (∃ k < S.length, S k < T k ∧ ∀ i < k, S i ≈ T i) ∨
    (T.length < S.length ∧ ∀ i < T.length, S i ≈ T i)
\end{leancode}

As in Definition~\ref{polyOrder}, we extend the relations \(\preceq\) and \(\sim\) to triangular sets. The relation \(\preceq\) is a preorder on triangular sets, and it induces a partial order on the quotient \(\mathrm{Triangular Set}/\!\sim\).

The polynomial order is well-founded whenever \(\sigma\) is well-ordered. For triangular sets, however, well-foundedness requires the stronger assumption that \(\sigma\) is \textbf{finite}.

\begin{proposition}[Well-ordering property of triangular sets]
\label{wellOrderingTriSets}
If \(\sigma\) is finite, then every sequence of triangular sets that is strictly decreasing with respect to the order \(\prec\) is finite.
\end{proposition}

\begin{leancode}
instance [Finite σ] : WellFoundedLT (TriangularSet σ R)
\end{leancode}

This well-ordering property is a key ingredient in proving the termination of the algorithms for computing characteristic sets and zero decompositions.

\begin{definition}[{\ourdef[Ascending Set]{TriangularSet.isAscendingSet}}]
\label{ascSet} %\cite{wen1986basic}
A triangular set \(AS = \{as_1, \ldots, as_r\}\) is called an \emph{ascending set} if, for every pair \(i,j\) with \(i < j\), the polynomial \(as_j\) is reduced with respect to \(as_i\).
\end{definition}

\begin{leancode}
class AscendingSetTheory (σ R : Type*) [CommSemiring R] [DecidableEq R] [LinearOrder σ] where
  /-- The reduction relation used to define the ascending property. -/
  protected reducedTo' : MvPolynomial σ R → MvPolynomial σ R → Prop
  decidableReducedTo : DecidableRel reducedTo' := by infer_instance
  /-- A key property linking the ascending set structure to the initial.
  If `S` is an ascending set, the initial of any non-constant element in `S`
  must be reduced with respect to `S`. -/
  protected initial_reducedToSet_of_max_vars_ne_bot : ∀ ⦃S : TriangularSet σ R⦄ ⦃i : ℕ⦄,
    (∀ ⦃i j⦄, i < j → j < S.length → reducedTo' (S j) (S i)) →
    (S i).vars.max ≠ ⊥ → (S i).initial.reducedToSet S

def isAscendingSet (S : TriangularSet σ R) : Prop :=
  ∀ ⦃i j⦄, i < j → j < S.length → AscendingSetTheory.reducedTo' (S j) (S i)
\end{leancode}

% Another definition of ascending set is based on the reduction of initials, i.e. $\init(as_j)$ is reduced with respect to $as_i$.

The essential property of an ascending set \(AS\) is that, for every nonconstant polynomial \(p \in AS\), its initial \(I_p\) is reduced with respect to \(AS\) itself. This property plays a crucial role in the termination proof of the zero decomposition algorithm.

\begin{definition}[{\ourdef[Basic Set]{HasBasicSet}}]
\label{basicSet}
A basic set $BS$ of a polynomial set $PS$ is the minimal ascending set contained in $PS$.
\end{definition}
\begin{leancode}
class HasBasicSet (σ R : Type*) [CommSemiring R] [DecidableEq R] [LinearOrder σ]
    extends AscendingSetTheory σ R where
  /-- Computes a Basic Set from a list of polynomials. -/
  basicSet : List (MvPolynomial σ R) → TriangularSet σ R
  basicSet_isAscendingSet (l : List (MvPolynomial σ R)) : (basicSet l).isAscendingSet
  basicSet_subset (l : List (MvPolynomial σ R)) : ∀ ⦃c⦄, c ∈ basicSet l → c ∈ l
  basicSet_minimal (l : List (MvPolynomial σ R)) :
      ∀ ⦃S⦄, S.isAscendingSet → (∀ ⦃p⦄, p ∈ S → p ∈ l) → basicSet l ≤ S
  /-- Order reduction property: appending a reduced element strictly decreases the basic set order.
  Crucial for proving termination of zero decomposition. -/
  basicSet_append_lt_of_exists_reducedToSet : ∀ ⦃l1 l2 : List (MvPolynomial σ R)⦄,
      (∃ p ∈ l2, p ≠ 0 ∧ p.reducedToSet (basicSet l1)) → basicSet (l2 ++ l1) < basicSet l1
\end{leancode}

The next proposition is used to show that, throughout the algorithm, the order of the basic set strictly decreases.

\begin{proposition}
Let \(p\) be a polynomial reduced with respect to \(PS\). Then the basic set of \(PS \cup \{p\}\) is strictly smaller than the basic set of \(PS\) with respect to the order on basic sets.
\end{proposition}

\begin{leancode}
theorem basicSet_append_lt_of_exists_reducedToSet
    (h : ∃ p ∈ l2, p ≠ 0 ∧ p.reducedToSet l1.basicSet) : (l2 ++ l1).basicSet < l1.basicSet
\end{leancode}

\section{Polynomial Pseudo-Division}

\begin{theorem}[{\ourdef[Pseudo-Remainder]{MvPolynomial.pseudo}}]
\label{remainder}
Let \(f,g \in R[X_i]_{i\in\sigma}\), and suppose that \(f\) is nonconstant with initial \(I=\init(f)\). Then there exist \(s\in\mathbb{N}\), \(q\in R[X_i]_{i\in\sigma}\), and \(r\in R[X_i]_{i\in\sigma}\) such that
\[
I^s \cdot g = q \cdot f + r,
\]
where \(r\) is reduced with respect to \(f\). The polynomial \(r\) is called the \emph{pseudo-remainder} of \(g\) with respect to \(f\), denoted by
\[
r=\remdr(g/f).
\]
The following procedure gives an algorithm for computing \(s\), \(q\), and \(r\).
\end{theorem}

\noindent\fbox{%
\begin{minipage}{\dimexpr\linewidth-2\fboxsep-2\fboxrule\relax}
\begin{algorithmic}[1]
\REQUIRE polynomial $g$, polynomial $f$ with its main variable $X_m$, $I = \init(f)$
\ENSURE remainder $r$ of $g$ divided by $f$
\STATE $s := 0$, $q := 0$, $r := g$
\STATE $d_f := \deg_m(f)$
\REPEAT
  \STATE $d_r := \deg_m(r)$
  \STATE $t := \init_m(r) \cdot X_m^{d_r - d_f}$
  \STATE $r := I \cdot r - t \cdot f$
  \STATE $q := I \cdot q + t$
  \STATE $s := s + 1$
\UNTIL {$d_r < d_f$}
\RETURN $r$
\end{algorithmic}
\end{minipage}%
}

The key step is Step~6, where \(r\) is multiplied by \(I\), while \(f\) is multiplied by \(\init_m(r)\), so that the leading term of \(r\) is canceled. The factor \(X_m^{d_r-d_f}\) aligns the degrees of \(r\) and \(f\) in the main variable \(X_m\). Consequently, the degree of \(r\) strictly decreases at each iteration.

\begin{leancode}
def pseudoOf.go (i : σ) (f : MvPolynomial σ R) (s : ℕ) (q r : MvPolynomial σ R)
    (h : f.degreeOf i ≠ 0) : PseudoResult (MvPolynomial σ R) :=
  if r.degreeOf i < f.degreeOf i then ⟨s, q, r⟩
  else
    letI d := r.degreeOf i
    letI Ic_r := r.initialOf i
    letI x_power := X i ^ (d - f.degreeOf i)
    let term := Ic_r * x_power
    let I := f.initialOf i
    let q' := I * q + term
    let r' := I * r - term * f
    go i f (s + 1) q' r' h
  termination_by r.degreeOf i

def pseudoOf (i : σ) (f : MvPolynomial σ R) : PseudoResult (MvPolynomial σ R) :=
  if h : f.degreeOf i = 0 then ⟨1, g, 0⟩
  else pseudoOf.go i f 0 0 g 

def pseudo (f : MvPolynomial σ R) : PseudoResult (MvPolynomial σ R) :=
  if f = 0 then ⟨0, 0, g⟩
  else
    match f.vars.max with
    | ⊥ => ⟨0, (f.coeff 0)⁻¹ • g, 0⟩
    | some c => g.pseudoOf c f

theorem pseudo_equation : f.initial ^ (g.pseudo f).exponent * g = (g.pseudo f).quotient * f + (g.pseudo f).remainder
\end{leancode}

This allows us to define the pseudo-remainder with respect to a triangular set.

\begin{theorem}[{\ourdef[Pseudo-remainder with respect to a triangular set]{MvPolynomial.setPseudo}}]
Let \(S=(s_1,\ldots,s_k)\) be a triangular set and let \(g \in R[X_i]_{i\in\sigma}\). Then there exist
\[
r, q_1,\ldots,q_k \in R[X_i]_{i\in\sigma}
\]
and \(e_1,\ldots,e_k \in \mathbb{N}\) such that
\[
\left(\prod_{i=1}^k \init(s_i)^{e_i}\right) g
=
\sum_{i=1}^k q_i s_i + r,
\]
where \(r\) is reduced with respect to \(S\).
 The polynomial \(r\) is called the \emph{pseudo-remainder} of \(g\) with respect to \(S\) and is denoted by
\[
r=\remdr(g/S).
\]
\end{theorem}

\begin{proof}
This remainder can be computed by recursively performing pseudo-division, starting from the last element of \(S\). First, we perform pseudo-division of \(g\) by \(s_k\), obtaining \(e_k\), \(q_k\), and \(r_k\). We then perform pseudo-division of \(r_k\) by \(s_{k-1}\), obtaining \(e_{k-1}\), \(q_{k-1}\), and \(r_{k-1}\). Continuing in this way, we eventually obtain a final remainder, \(r_0\). This polynomial \(r_0\) is the pseudo-remainder of \(g\) with respect to \(S\).

\end{proof}

\begin{leancode}
def setPseudo.go (f : ℕ → MvPolynomial σ R) (fuel : ℕ) (es : List ℕ)
    (qs : List (MvPolynomial σ R)) (r : MvPolynomial σ R) : SetPseudoResult (MvPolynomial σ R) :=
  if fuel = 0 then ⟨es, qs, r⟩
  else
    let p := r.pseudo (f (fuel - 1))
    let es' := p.exponent :: es
    let qs' := p.quotient :: qs.map (· * (f (fuel - 1)).initial ^ p.exponent)
    let r' := p.remainder
    go f (fuel - 1) es' qs' r'

def setPseudo (S : TriangularSet σ R) : SetPseudoResult (MvPolynomial σ R) :=
  setPseudo.go S S.length [] [] g

theorem setPseudo_equation : letI result := g.setPseudo S
    (Π i : Fin result.exponents.length, (S i).initial ^ result.exponents[i]) * g
    = (∑ i : Fin result.quotients.length, result.quotients[i] * S i) + result.remainder
\end{leancode}

\section{Characteristic Set}

\begin{definition}[{\ourdef[Characteristic Set]{TriangularSet.isCharacteristicSet}}]
\label{charSet}
A triangular set \(CS\) is called a \emph{characteristic set} of a polynomial set \(PS\) if it satisfies the following conditions:
\begin{enumerate}
  \item for every \(p \in PS\), \(\remdr(p/CS)=0\);
  \item \(\zero(PS) \subseteq \zero(CS)\).
\end{enumerate}
\end{definition}

\begin{leancode}
def TriangularSet.isCharacteristicSet [CommSemiring K] [Algebra R K]
    (CS : TriangularSet σ R) (a : α) : Prop :=
  (∀ g ∈ a, (0 : MvPolynomial σ R).isSetRemainder g CS) ∧ vanishingSet K a ⊆ vanishingSet K CS
\end{leancode}

\begin{theorem}[{\ourdef[Well-ordering Principle]{CharacteristicSet.vanishingSet_diff_initialProd_subset}}]
\label{wellOP}

Let \(CS = (c_1,\ldots,c_r)\) be a characteristic set of \(PS\), and let
\(
IP=\prod_{i=1}^r I_{c_i}
\)
be the initial product of \(CS\). Then
\[
\zero(CS/IP) \subseteq \zero(PS),
\]
and consequently
\[
\zero(CS/IP)=\zero(PS/IP).
\]
Moreover,
\begin{equation}\label{wellOP3}
\zero(PS)
=
\zero(CS/IP)
\cup
\left(\bigcup_{p\in CS}\zero(PS\cup\{I_p\})\right).
\end{equation}
\end{theorem}

\begin{leancode}
theorem vanishingSet_diff_initialProd_subset
    (h : (∀ g ∈ PS, (0 : MvPolynomial σ R).isSetRemainder g CS)) :
    vanishingSet K CS \ vanishingSet' K (initialProd CS.toFinset) ⊆ vanishingSet K PS

theorem vanishingSet_diff_initialProd_eq (h : CS.isCharacteristicSet K PS) :
    vanishingSet K CS \ vanishingSet' K (initialProd CS.toFinset) =
      vanishingSet K PS \ vanishingSet' K (initialProd CS.toFinset)

theorem vanishingSet_decomposition (h : CS.isCharacteristicSet K PS) : vanishingSet K PS =
      vanishingSet K CS \ vanishingSet' K (initialProd CS.toFinset) ∪
      (∪ p ∈ CS, vanishingSet K PS ∩ vanishingSet' K p.initial)
\end{leancode}

\begin{theorem}
A characteristic set $CS$ can be constructed in a finite number of steps using the following algorithm:
\end{theorem}

\noindent\fbox{%
\begin{minipage}{\dimexpr\linewidth-2\fboxsep-2\fboxrule\relax}
\begin{algorithmic}[1]
\REQUIRE  \( PS = (p_1, \ldots, p_s) \)
\ENSURE  a Characteristic set $CS = (c_1, \ldots, c_t)$ for \(PS\),
\STATE \( PS' := PS \)
\REPEAT
\STATE \( RS := \emptyset \)
\STATE \( BS := \text{BasicSet}(PS') \)
\FOR{each  \(\{p\}\) in \( PS' \)}
\STATE \( r := \remdr(p, BS) \)
\IF{\( r \neq 0 \)}
\STATE \( RS := RS \cup \{r\} \)
\ENDIF
\ENDFOR
\STATE \( PS' := PS \cup RS \cup BS \)
\UNTIL{\( R = \emptyset \)}
\RETURN \( CS := BS \)
\end{algorithmic}
\end{minipage}%
}

\begin{leancode}
def characteristicSet.go [Finite σ] (l₀ l : List (MvPolynomial σ R))
    : TriangularSet σ R :=
  let BS := l.basicSet.val
  let lBS := BS.toList
  let RS : List _ := ((l \ lBS).map fun p ↦ (p.setPseudo BS).remainder).filter (· ≠ 0)
  if RS = [] then BS
  else go l₀ (l₀ ++ RS ++ lBS)
  termination_by l.basicSet

def characteristicSet : TriangularSet σ R := characteristicSet.go l l
\end{leancode}

\begin{theorem}[{\ourdef[Zero Decomposition Theorem]{CharacteristicSet.zeroDecomposition}}]
\label{zeroDecomp}
For a polynomial system \(PS\), there exists a finite collection of triangular sets
\[
\mathcal{ZD}=\{CS_j\}_j
\]
such that
\begin{equation}\label{zeroDecomp1}
    \zero(PS)=\bigcup_j \zero(CS_j/IP_j),
\end{equation}
where \(IP_j\) denotes the initial product of \(CS_j\). Moreover, for every \(j\) and every \(p\in PS\),
\[
    \remdr(p/CS_j)=0.
\]
\end{theorem}

\begin{proof}
The collection \(\mathcal{ZD}\) is computed by the following recursive algorithm. Given a polynomial set \(PS\), we first compute a characteristic set \(CS\) of \(PS\). If \(CS\) contains a nonzero constant, then we set
\(
\mathcal{ZD}=\{CS\}.
\)
Otherwise, for each polynomial \(p\in CS\), we recursively compute a zero decomposition of
\[
PS\cup CS\cup\{I_p\}.
\]
The final collection \(\mathcal{ZD}\) is obtained by taking the union of all these recursive decompositions together with \(CS\) itself.

The identity~\eqref{zeroDecomp1} follows by recursively applying~\eqref{wellOP3}. Indeed, at each step, \(\zero(PS)\) is decomposed into the component \(\zero(CS/IP)\) and the components corresponding to the vanishing of one of the initials \(I_p\).

It remains to prove the remainder condition. Each triangular set \(CS_j\) appearing in the final decomposition is a characteristic set of some polynomial set
\[
PS' = PS\cup CS\cup\{I_p\}
\]
arising in the recursive construction. Since \(PS\subseteq PS'\), Definition~\ref{charSet} gives
\[
\remdr(p/CS_j)=0
\]
for every \(p\in PS\). This proves the second property.
\end{proof}

\begin{leancode}
def zeroDecomposition (l : List (MvPolynomial σ R)) : List (TriangularSet σ R) :=
  let CS := l.characteristicSet
  let subDecomp := (CS.toList.filter fun p ↦ p.vars.max ≠ ⊥).attach.map
    fun ⟨p, _⟩ ↦ zeroDecomposition (p.initial :: CS.toList ++ l)
  CS :: subDecomp.flatten
\end{leancode}

\begin{leancode}
theorem vanishingSet_eq_zeroDecomposition_union :
    vanishingSet K l = ∪ CS ∈ l.zeroDecomposition,
      vanishingSet K CS \ vanishingSet' K (initialProd CS.toFinset)
\end{leancode}

\section{Conclusions}\label{sec4}

In this paper, we present a machine-checked formalization of the Wu-Ritt characteristic set method in Lean~4. Our development covers polynomial orders, pseudo-division, triangular and ascending sets, basic and characteristic sets, and zero decomposition, with proofs of termination and correctness. In particular, we verified the well-ordering principle and the zero-set decomposition theorem.

This work provides a reusable formal foundation for certified polynomial system solving and mechanical geometry theorem proving in Lean. Future work includes executable code extraction and extensions to differential polynomials.

\begin{credits}
\subsubsection{\ackname}This work was supported by the National Key R\&D Program of China 2023YFA1009401 and  the Strategic Priority Research Program of Chinese Academy of Sciences under Grant XDA0480502. Junyu Guo is encouraged and supported by his supervisor, Xishun Zhao, to work on this project.
\end{credits}

% \begin{credits}
% \subsubsection{\ackname} A bold run-in heading in small font size at the end of the paper is
% used for general acknowledgments, for example: This study was funded

% \subsubsection{\discintname}
% It is now necessary to declare any competing interests or to specifically
% state that the authors have no competing interests. Please place the
% statement with a bold run-in heading in small font size beneath the
% (optional) acknowledgments\footnote{If EquinOCS, our proceedings submission
% system, is used, then the disclaimer can be provided directly in the system.},
% for example: The authors have no competing interests to declare that are
% relevant to the content of this article. Or: Author A has received research
% grants from Company W. Author B has received a speaker honorarium from
% Company X and owns stock in Company Y. Author C is a member of committee Z.
% \end{credits}


\bibliographystyle{splncs04}
\bibliography{main}

\end{document}